\documentclass[a4paper, 11pt, reqno]{amsart}
\usepackage{enumerate}

\usepackage{hyperref}
\hypersetup{colorlinks=true,linkcolor=blue,citecolor=blue,pdfpagemode=UseNone,pdfstartview={XYZ null null 1.00}}

\usepackage{amsmath, amsthm, amsopn, amssymb}
\usepackage{xcolor}
\usepackage{graphicx}

\usepackage{geometry}
\theoremstyle{plain}
\newtheorem*{theorem*}{Theorem}
\newtheorem{theorem}{Theorem}[section]
\newtheorem{definition}{Definition}[section]
\newtheorem{lemma}[theorem]{Lemma}

\newtheorem{proposition}[theorem]{Proposition}
\newtheorem*{claim*}{Claim}

\newtheorem{question}[theorem]{Question}
\theoremstyle{remark}

\newcommand{\norm}[1]{\left\lVert#1\right\rVert}

\def\VV{\mathbb{V}}

\DeclareMathOperator\Tourn{Tourn}
\DeclareMathOperator\guar{guar}
\DeclareMathOperator\outdeg{outdeg}
\DeclareMathOperator\rand{rand}

\let\eps\varepsilon

\let\originalleft\left
\let\originalright\right
\renewcommand{\left}{\mathopen{}\mathclose\bgroup\originalleft}
\renewcommand{\right}{\aftergroup\egroup\originalright}

\makeatletter
\def\imod#1{\allowbreak\mkern10mu({\operator@font mod}\,\,#1)}
\makeatother

\title{The performance guarantee of randomized perfect voting trees}

\author{Jason Long} \address{Department of Pure Mathematics and
Mathematical Statistics, University of Cambridge, Wilberforce Road,
Cambridge CB3\thinspace0WB, UK} \email{jl694@cam.ac.uk}

\author{Adam Zsolt Wagner} \address{Department of Mathematics, ETH, R\"amistrasse 101, 8092 Z\"urich, Switzerland }
\email{zsolt.wagner@math.ethz.ch}

\begin{document}

\begin{abstract} In this note we study randomized voting trees, previously
introduced by Fisher, Procaccia and Samorodnitsky~\cite{FPS}. They speculate that a
non-trivial performance guarantee may be achievable using randomized,
balanced trees whose height is carefully chosen. We explore some connections to the so-called Volterra 
quadratic stochastic operators, and show that uniformly random voting
trees cannot provide a performance guarantee that is linear in the number of individuals.
\end{abstract} 
\maketitle 
\section{Introduction} 
A
well-known question in social choice theory is the following: given a
collection of $n$ candidates which are each pairwise comparable, how should
we select a winner? The pairwise comparisons between candidates may be
encoded as a tournament of $n$ vertices, with edges directed from the
winner to the loser in each comparison. The subtlety of the problem lies in
the fact that the tournament need not be transitive, and therefore there is
no indisputable way to select a global winner. In order to evaluate the candidates, some scoring system must be chosen which assigns scores to the candidates in a manner which may depend on the tournament. A common approach is to compute their \emph{Copeland score}, which is simply the
number of other candidates that they beat, or the out-degree of the vertex
in the tournament.

One natural way of selecting a winning candidate from a tournament is to
use a \emph{voting tree}. Voting trees were first introduced by Farquharson~\cite{farq} and extensively studied in e.g.~\cite{vr1,vr2,vr3,vr4,vr5,vr6,vr7,vr8}. A voting tree is a complete binary tree with
leaves labelled from $[n]$ (with possible repeats). Given any tournament as input, a voting tree deterministically selects a winner from
the tournament by running pairwise elections between the leaves until we
arrive at the root node. Let us  denote by $\Tourn(n)$ the set of all tournaments on $n$ labelled vertices. Hence voting trees can be regarded as functions from $\Tourn(n)$ to $[n]$, and we write $\Theta(T)$ for the winner of the tournament $T$ under the voting tree $\Theta$. Using the Copeland score to evaluate the candidates, it makes sense to define the \emph{performance guarantee} $\guar(\Theta)$ of a voting tree $\Theta$, which is the minimum out-degree of any winner
that it produces. That is,
$$\guar(\Theta) = \min\{\outdeg\left(\Theta(T)\right):T\in\Tourn(n)\},$$
where $\outdeg(\Theta(T))$ denotes the number of candidates in $T$ that are beaten by $\Theta(T)$.

Recent work has focussed on determining the largest possible performance
guarantee for a voting tree. The trivial upper bound is $\lfloor
(n-1)/2\rfloor$, which would be achieved if the voting tree returned the vertex of maximal out-degree for any input tournament. For a lower bound, we might first consider the voting tree
of depth $\log_2(n)$, with $n$ a power of 2, and leaves labelled from $1$
to $n$ in any order. It is easy to see that a winner produced by such a tree has
out-degree at least $\log_2(n)$, since precisely this many elections are
won against distinct opponents on the way to the root node, but it is also
possible to construct tournaments for which the winner has out-degree
precisely $\log_2(n)$ for this voting tree. Hence the performance guarantee of this balanced tree is precisely $\log_2 (n)$.

Both of these bounds have been recently improved. In 2011, Fischer,
Procaccia and Samorodnitsky~\cite{FPS} showed that no voting tree can guarantee an
out-degree larger than $3/4+o(1)$ times the maximum out-degree in any input
tournament, improving the trivial upper bound of $n/2$ to $3n/8+o(n)$. Then in 2012, Iglesias,
Ince and Loh~\cite{poshen} achieved a significant breakthrough with a construction of a voting tree achieving a guaranteed
out-degree of $(\sqrt{2}+o(1))\sqrt{n}$. Their explicit construction is based on an ingenious recursion.

Fisher, Procaccia and Samorodnitsky also consider the problem of
determining the performance of a randomized voting tree (where the depth is
fixed but the labels for the leaves are chosen according to some
distribution $\Delta$). 
They study the specific distribution obtained by uniformly
labelling the vertices of a complete balanced voting tree of depth $d$ with labels from $[n]$. Such a voting tree is
called a $d$-RPT (random perfect tree) and will be denoted by $\Theta^{\rand}_{d,n}$. It is natural to think that for a fixed $n$,
choosing a $d$-RPT with large $d$ may be a good choice of voting tree, since
the winner will have won many elections and the collection of opponents
that the winner meets is, at least in the early matches, fairly random.
They show, however, that this intuition is flawed, since for any  fixed $n$ and any small $\eps>0$
there exist arbitrarily large depths $d$ for which the guarantee of the  $d$-RPT is $1$ with probability at least $1-\eps$. Despite this, they speculate that
it may be the case that, for a given $n$, one can select arbitrarily large depths $d=d(n)$
for which the $d$-RPT can provide some sort of approximation guarantee.

\begin{question}\label{main} Let $n\in\mathbb{N}$
and let $\epsilon>0$. What is the largest out-degree that we can obtain
with probability at least $1-\epsilon$ for a $d$-RPT  with a suitable choice of
$d=d(n,\epsilon)$? \end{question}

An out degree of $f(n)$ in Question~\ref{main}, say,
would prove the existence of a deterministic voting tree with an
approximation guarantee of $f(n)$, simply by taking
$\epsilon$ sufficiently small.

In this note, we explore this question in more detail, uncovering links
to some well-studied dynamical systems. By using results on the so-called
\emph{Volterra systems} we show that the $d$-RPT cannot be used to give a
linear out-degree for all $n$.

\begin{theorem}\label{theorem} For any $\eps>0$ there exist infinitely many $n$ with the property that
$$\limsup_{d\rightarrow\infty}\mathbb{P}\left(\guar(\Theta^{\rand}_{d,n})>\eps n\right)\leq \eps.$$
 \end{theorem}

The above theorem essentially says that if $n$ is constant then the random voting tree $\Theta^{\rand}_{d,n}$ has sublinear guarantee for all sufficiently large $d$. It is a natural question to ask what the relation between $\eps$ and the smallest allowed choice for $n$ is in the statement of the theorem. We note that by examining the proof and doing all calculations very carefully we could show that this least $n$ is bounded above by $\exp\left(\exp\left(\ldots\exp\left(\eps^{-1}\right)\ldots\right)\right)$ where there are seven ``$\exp$''-s. In particular this means that for infinitely many $n$ we may take 
$$\eps=\frac{1}{\underbrace{\log(\log(\ldots\log}_{7\text{-fold iterated }\log}   (n)\ldots))}.$$

Since these calculations are quite tedious and the $7$-fold logarithm is likely far from optimal, we present the proof of Theorem~\ref{theorem} without providing bounds on $n$ and hence $\eps$. We are not sure what the correct behaviour of $\eps$ should be -- in fact we cannot even answer the following question:

\begin{question}
Is the following statement true or false? If $n$ is an arbitrary fixed integer then as $d\rightarrow \infty$ we have
$$\mathbb{P}\left(\guar(\Theta^{\rand}_{d,n})\leq10\right)\rightarrow 1.$$
\end{question}

Hence, this theorem does not eliminate the interest in
Question~\ref{main}. Any positive result giving out-degree larger than
$\mathcal{O}(\sqrt{n})$ would improve on the result of Iglesias, Ince and
Loh, and we leave a large gap between this and our upper bound. Progress
with such an approach, however, appears likely to require significant
advances in our understanding of the relevant dynamical systems.

Our proof of Theorem~\ref{theorem} is based on the connections between the $d$-RPT and the Stein-Ulam Spiral which were established in~\cite{FPS}. Our proof makes heavy use of the extensive literature on this dynamical system, in particular the results of~\cite{misi, zakhar}.

\section{The Volterra QSO}

We begin our analysis of the $d$-RPT by highlighting a connection to a
dynamical system, following~\cite{FPS}.

Let our candidates be elements of $[n]$ for some fixed $n$, and let our
voting tree $\Theta$ be a $d$-RPT, that is, a binary tree with height $d$ and $2^d$ leaves, each leaf receiving a label from $[n]$ uniformly at random independently from other leaves. Fix some tournament $T\in\Tourn (n)$ and let $P_i(k)$ denote the probability that a vertex
at depth $k$ gets labelled with candidate $i$ when $\Theta$ is applied to the tournament $T$. At depth $d$, the vertices of
our voting tree are labelled uniformly at random, and so we have
$P_i(d)=1/n$ for all $i$. As we move up the voting
tree these probabilities evolve according to a quadratic dynamical system.

In particular, given a candidate $i$ which is beaten by candidates in the
set $A\subset [n]$ and which beats the set $B\subset [n]$, we see that a
given vertex $v$ at depth $k-1$ in our voting tree is labelled $i$ if and
only if one of the children of $v$ is labelled $i$ and the other child is
not labelled from $A$. This gives that
\[P_i(k-1)=P_i(k)\left(P_i(k)+2\sum_{j\in B}P_j(k)\right)\]
\[=P_i(k)\left(1-\sum_{j\in A}P_j(k)+\sum_{j\in B}P_j(k)\right).\] This
quadratic stochastic operator is precisely an instance of the Volterra Quadratic Stochastic Operator (QSO).
\begin{definition} Fix $n$ and let $\mathbf{x}(0)=(1/n,\dots,1/n)$ be an $n$-tuple.
For a tournament $T$ and an $n$-tuple $\mathbf{x}=(x_1,\dots,x_n)$ we define
$V_T(\mathbf{x})$ to be \[x_i\left(1-\sum_{j\in A}x_j+\sum_{j\in B}x_j\right)\]
where $A$ and $B$ are as above. Let $\mathbf{x}(t+1)=V_T(\mathbf{x}(t))$. \end{definition}
Observe that $\mathbf{x}(d-1)$ is the probability distribution of the winning vertex
when tournament $T$ is fed into the $d$-RPT voting tree.

The behaviour of the QSO $V_T$ can be complex even when $T$ is simple, see e.g.~\cite{gani3, gani2, sabu1}. We will prove Theorem~\ref{theorem} by studying the simplest non-trivial QSO $V_T$ where $T$ is the 3-vertex cyclic tournament;
candidate 1 is beaten by candidate 2, who is beaten by candidate 3, who in
turn is beaten by candidate 1. Let us refer to this QSO, also called the Stein-Ulam Spiral (see~\cite{misi}), simply as $\VV$. It is given by the recursion
\begin{align*}
\VV(x,y,z) &=(x^2+2xy,   y^2+2yz, z^2+2zx) \\
&= (x(1+y-z), y(1+z-x), z(1+x-y)).
\end{align*}

\section{Properties of the system $\VV$}

The Stein-Ulam Spiral was first studied by Stan Ulam and Paul Stein in the 1950's, using the computers in the Los Alamos
National Laboratory~\cite{alamos}. Due to its non-typical behaviour this system and its long-term behaviour has received considerable attention throughout the years and has been extensively studied, see e.g.~\cite{misi,FPS,gani1, biobook, alamos,zakhar}. 

Define the $2$-simplex $\Delta$ as $\Delta = \{(x,y,z)\in\mathbb{R}^3:0\leq x,y,z\leq 1, x+y+z=1\}$. Recall that the map $\VV:\Delta\rightarrow \Delta$ is defined as follows:
$$\VV(x,y,z)=\left(x(1+y-z),y(1+z-x), z(1+x-y)\right).$$

$\VV$ is a bijective, invertible map from the 2-simplex to itself. It has fixed points at the
three vertices $(1,0,0)$,$(0,1,0)$,$(0,0,1)$ and at the point $(1/3,1/3, 1/3)$. Any point $\mathbf{x}$ not equal
to $(1/3,1/3, 1/3)$ in the interior of the simplex has a limit set which
contains all three corners. Moreover, the orbit of the point $\mathbf{x}$ under
$\VV$ moves round the corners of the simplex in a cyclic fashion, spending
(for all $\epsilon>0$) a proportion of more than $1-\epsilon$ of all steps
within a distance $\epsilon$ of each corner.  Defining the potential function $\phi:\Delta\rightarrow \mathbb{R}$ by $\phi((x,y,z))=xyz$ it is easy to show that   $\phi(\mathbf{a})\geq \phi(\VV(\mathbf{a}))$ for all $\mathbf{a}\in\Delta$. Figure~\ref{alamospicture} illustrates the behaviour of the system. It is taken from the original Los Alamos report~\cite{alamos}.

\begin{figure}[h!]
\centering
\includegraphics[width=0.8\textwidth]{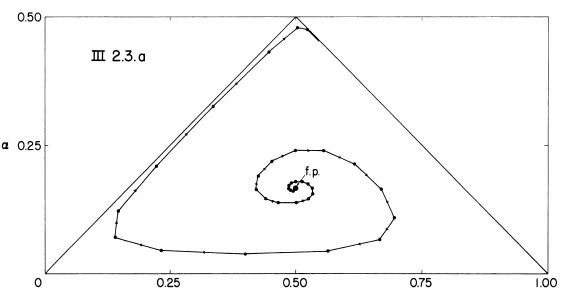}
\caption{The trajectory of a point under $\VV$}
\label{alamospicture}
\end{figure}

It is apparent from Figure~\ref{alamospicture} that the potential $\phi(\mathbf{x})$ decreases rapidly in successive iterations of $\VV$, hence any point gets very close to the boundary of $\Delta$ quickly. This means that one cannot afford even the smallest rounding errors when calculating values of $\VV^i(\mathbf{x})$, making the analysis of the long-term behaviour of $\VV$ via computer simulation difficult. The following basic observation will be used several times throughout this paper.
\begin{proposition}\label{easybounds}
Let $\mathbf{a}$ be a point in $\rm{int}(\Delta)$ with coordinates $(x,y,z)$ and let $\VV(\mathbf{a})$ have coordinates $(x',y',z')$. Then the following inequalities hold:
\begin{itemize}
\item $x'\leq 2x$ and $(1-x)^2\leq 1-x'$,
\item $y'\leq 2y$ and $(1-y)^2\leq 1-y'$,
\item $z'\leq 2z$ and $(1-z)^2\leq 1-z'$.
\end{itemize}
\end{proposition}
\begin{proof}
By symmetry it suffices to show the two inequalities in the first line. Note that $x'=x(1+y-z)\leq x(1+y)\leq 2x$, proving the first inequality. For the second inequality, we have
$$x'=x(1+y-z)\leq x(1+y)\leq x(1+(1-x))=1-(x-1)^2.$$
\end{proof}

Given a point $\mathbf{a}$ in the $2$-simplex $\Delta$, define its \emph{orbit} $\mathcal{O}(\mathbf{a})$ as $$\mathcal{O}(\mathbf{a}):=\{\VV^i(\mathbf{a}):i\in\mathbb{Z}\}.$$ If $\mathbf{a}$ has coordinates $(x,y,z)$ then define its \emph{rotation} as $R(\mathbf{a})=(y,z,x)$. Two points $\mathbf{a},\mathbf{b}\in \Delta$ are \emph{rotated} if either $R(\mathbf{a})=\mathbf{b}$ or $R(\mathbf{b})=\mathbf{a}$. Two points $\mathbf{a},\mathbf{b}\in \Delta$ are \emph{weakly rotated} if there exists $i\in\mathbb{Z}$ such that $\VV^i(\mathbf{a})$ and $\mathbf{b}$ are rotated.  If $\mathbf{a},\mathbf{b}$ are weakly rotated define their \emph{rotation distance} $\rm{rtd}(\mathbf{a},\mathbf{b})$ as $|i|$, where $i\in\mathbf{Z}$ is minimal such that $\VV^i(\mathbf{a})$ and $\mathbf{b}$ are rotated.

Note that if $\mathbf{a},\mathbf{b}$ are weakly rotated then $R\left(\mathcal{O}(\mathbf{a})\right)=\mathcal{O}(\mathbf{b})$ or $R\left(\mathcal{O}(\mathbf{b})\right)=\mathcal{O}(\mathbf{a})$. 
For any $\eps>0$ define the set $M(\eps)=\left\{(x,y,z)\in\Delta: \norm{ \bigl(x,y,z\bigr) - \bigl(\frac13,\frac13,\frac13\bigr)    }_{\infty}\geq \eps\right\}$, where $\norm{\cdot}_{\infty}$ denotes the sup norm on $\mathbb{R}^2$. Say that a point $\mathbf{a}\in\Delta$ \emph{is $\eps$-close to the $x$ corner} if $\norm{\mathbf{a}-(1,0,0)}_{\infty}\leq \eps$. The definitions for being $\eps$-close to the $y$ or $z$ corner are analogous. We say that a point $\mathbf{a}\in\Delta$ is $\eps$-close to a vertex of $\Delta$ if $\mathbf{a}$ is $\eps$-close to either the $x$, $y$ or $z$ corner. 

Similarly, say that a point is \emph{$\eps$-close to the $xy$ side} if its $z$ coordinate is at most $\eps$, and the definitions for being $\eps$-close to the $yz$ and $xz$ sides are analogous.
To prove our main result we will need the following four propositions. Their proofs are straightforward but somewhat technical, so we postpone the details until the next section.
\begin{proposition}\label{99decr}
For all $\eps>0$ sufficiently small, for any $\mathbf{a}\in M(\eps)$ we have $\phi(\VV(\mathbf{a}))\leq (1-\eps^{3})\phi(\mathbf{a})$.
\end{proposition}

\begin{proposition}\label{epsclosecompact}
For all $\eps >0$ there exists a constant $D=D(\eps)$ such that for all $\mathbf{a}\in M(\eps)$ there is an integer $f(\mathbf{a})$ with $0\leq f(\mathbf{a})\leq D$ so that $\VV^{f(\mathbf{a})}(\mathbf{a})$ is $\eps$-close to a vertex of $\Delta$.
\end{proposition}

\begin{proposition}\label{skipcorner} For every $0<\eps<0.1$ and integer $D\in\mathbb{N}$ there exists an $\eps'$ with $0<\eps'<\eps$ such that the following statement holds. If $\mathbf{a}\in\Delta$ is an arbitrary point that is $\eps'$-close to the $xy$ side but not $\eps$-close to the $x$ corner, and $d$ is positive integer with the property that $\VV^d(\mathbf{a})$ is $\eps'$-close to the $xz$ side, then $d\geq D$.
\end{proposition}

\begin{proposition}\label{skipcorner2} For every $0<\eps<0.1$ and integer $D\in\mathbb{N}$ there exists an $\eps'$ with $0<\eps'<\eps$ such that the following statement holds. If $\mathbf{a}\in\Delta$ is an arbitrary point that is $\eps'$-close to the $x$ corner of $\Delta$, then for all $i\in \{0,1,\ldots,D\}$, the point $\VV^i(\mathbf{a})$ is $\eps$-close to the $x$ corner of $\Delta$. 
\end{proposition}

Recall that whenever $\mathbf{a}\in\rm{int}(\Delta)\setminus(1/3,1/3,1/3)$ the limit set of its orbit contains the three vertices of $\Delta$. For an $\eps>0$ define $d_x(\mathbf{a},\eps)$ (and $d_y(\mathbf{a},\eps),d_z(\mathbf{a},\eps)$) to be the least positive integer such that $\VV^{d_x(\mathbf{a},\eps)}(\mathbf{a})$ is $\eps$-close to the $x$ corner (and $y$ corner, $z$ corner resp.). The key ingredient to our proof of Theorem~\ref{theorem} is the following Lemma:

\begin{lemma}\label{sixpoints}
For all rational $\eps>0$ there exist six points $\mathbf{a},\mathbf{b},\mathbf{c},\mathbf{A},\mathbf{B},\mathbf{C}\in\rm{int}(\Delta)$ and an integer $d_0=d_0(\eps)$ such that the following two conditions hold:
\begin{itemize}
\item $\mathbf{a},\mathbf{b},\mathbf{c},\mathbf{A},\mathbf{B},\mathbf{C}$ all have rational coordinates and are $\eps$-close to the $x$ corner, 
\item for all $d\geq d_0$ at least one of the points $\VV^d(\mathbf{a}),\ldots,\VV^d(\mathbf{C})$ is $\eps$-close to the $x$-corner.
\end{itemize}
\end{lemma}
\begin{proof}
Let $$\mathbf{a}=(\eps/2,\eps/2,1-\eps),\quad \mathbf{b}=\VV^{d_x(R(\mathbf{a}),\eps)}(R(\mathbf{a})),\quad \mathbf{c}=\VV^{d_x(R^2(\mathbf{a}),\eps)}(R^2(\mathbf{a})),$$
so that $\mathbf{a},\mathbf{b},\mathbf{c}$ are pairwise weakly rotated and all $\eps$-close to the $x$ corner. Let their pairwise rotational distances be $i_1,i_2,i_3$ and let $D_1=\max\{i_1,i_2,i_3\} = \max \{d_x(R(\mathbf{a}),\eps),d_x(R^2(\mathbf{a}),\eps) \}$.

Let $\eps_1>0$ be given by Proposition~\ref{skipcorner2} with parameters $\eps, D_1$. Let $D_2$ be given by Proposition~\ref{epsclosecompact} with parameter $\eps_1$. 
 Let $\mathbf{A},\mathbf{B},\mathbf{C}\in\Delta$ be defined as
$$\mathbf{A}=\VV^{D_2+d_x\left(\VV^{D_2}(\mathbf{a}),\eps\right)}(\mathbf{a}),$$
$$\mathbf{B}=\VV^{D_2+d_x\left(\VV^{D_2}(\mathbf{b}),\eps\right)}(\mathbf{b}),$$
$$\mathbf{C}=\VV^{D_2+d_x\left(\VV^{D_2}(\mathbf{c}),\eps\right)}(\mathbf{c}),$$
so that $\mathbf{A}\in\mathcal{O}(\mathbf{a}),\mathbf{B}\in\mathcal{O}(\mathbf{b})$ and $\mathbf{C}\in\mathcal{O}(\mathbf{c}).$ Observe that $\mathbf{a},\mathbf{b},\mathbf{c},\mathbf{A},\mathbf{B},\mathbf{C}$ are all $\eps$-close to the $x$ corner.
Let $D_3:=D_2 + d_x\left(\VV^{D_2}(\mathbf{a}),\eps\right) + d_x\left(\VV^{D_2}(\mathbf{b}),\eps\right) + d_x\left(\VV^{D_2}(\mathbf{c}),\eps\right)$. Finally, let $\eps_2$ be given by Proposition~\ref{skipcorner} with parameters $\eps_1,D_3$.

 Let $d_0>0$ be such that $\phi(\VV^{d_0}(\mathbf{a}))<\eps_2^3$ and observe that this implies that $\phi(\VV^{d_0}(\mathbf{b}))$, $\phi(\VV^{d_0}(\mathbf{c}))$, $\phi(\VV^{d_0}(\mathbf{A}))$, $\phi(\VV^{d_0}(\mathbf{B}))$, $\phi(\VV^{d_0}(\mathbf{C}))<\eps_2^3$ and hence all six points are $\eps_2$-close to some side of~$\Delta$. Note that $d_0$ can be taken to only depend on $\eps$ -- this can be seen both by Proposition~\ref{99decr}, and by the fact that all three points $\mathbf{a},\mathbf{b},\mathbf{c}$ and hence all parameters only depend on $\eps$. Now fix some $d\geq d_0$.

Recall that our goal is to show that one of the six points  $\VV^d(\mathbf{a}),\ldots,\VV^d(\mathbf{C})$ is $\eps$-close to the $x$-corner. We split into cases according to which side or corner of $\Delta$ the point $\VV^d(\mathbf{a})$ is close to. Assume first that $\VV^d(\mathbf{a})$ is $\eps_1$-close to the $y$ corner of $\Delta$. Note that $$\VV^d(\mathbf{b})=\VV^d\left(\VV^{d_x(R(\mathbf{a}),\eps)}(R(\mathbf{a}))\right)=R\left(\VV^{d_x(R(\mathbf{a}),\eps)}\left(\VV^d(\mathbf{a})\right)\right),$$  
and hence
$$R^2\left(\VV^d(\mathbf{b})\right)=\VV^{d_x(R(\mathbf{a}),\eps)}\left(\VV^d(\mathbf{a})\right).$$

As $d_x(R(\mathbf{a}),\eps)\leq D_1$, by our choice of $\eps_1$ this implies that $R^2\left(\VV^d(\mathbf{b})\right)$ is $\eps$-close to the $y$ corner of $\Delta$ and hence $\VV^d(\mathbf{b})$ is $\eps$-close to the $x$ corner of $\Delta$, as required. The case where $\VV^d(\mathbf{a})$ is $\eps_1$-close to the $z$ corner of $\Delta$ is very similar, with the conclusion being that then $\VV^d(\mathbf{c})$ is $\eps$-close to the $x$ corner of $\Delta$.

Now assume that $\VV^d(\mathbf{a})$ is $\eps_2$-close to the $xy$ side but is not $\eps_1$-close to any corner. Consider the location of $\VV^d(\mathbf{A})$ and note that $\VV^d(\mathbf{A})=\VV^{D_2+d_x\left(\VV^{D_2}(\mathbf{a}),\eps\right)}\left(\VV^d(\mathbf{a})\right)$. By the choice of $D_2$ we know that there exists some $i$ with $0<i\leq D_2$ such that $\VV^{d+i}(\mathbf{a})$ is $\eps_1$-close to the $x$ corner. By the choice of $\eps_2$ and $D_3$, we have that for all $j$ with $0\leq j \leq D_3$, the point $\VV^{d+j}(\mathbf{a})$ is not $\eps_2$-close to the $xz$ side. This implies that for all $j$ with $i\leq j \leq D_3$, the point $\VV^{d+j}(\mathbf{a})$ is $\eps_1$-close to the $x$ corner, and hence in particular $\VV^d(\mathbf{A})$ is $\eps$-close to the $x$ corner, as required. 

The other two cases, where we assume that $\VV^d(\mathbf{a})$ is $\eps_2$-close to the $yz$ side or the $xz$ side are very similar, with the conclusion being that then $\VV^d(\mathbf{B})$ or $\VV^d(\mathbf{C})$ is $\eps$-close to the $x$ corner. This finishes the proof that for all $d\geq d_0$, at least one of the six points is $\eps$-close to the $x$ corner.
\end{proof}

\section{Proof of the main result}

Instead of proving Theorem~\ref{theorem} as stated, we will in fact prove a slightly stronger statement. We will show that we cannot achieve linear guarantee even if we restrict ourselves to \emph{tripartite tournaments}. A tripartite tournament, for the purposes of the present paper, is a tournament $T$ whose vertex set $V(T)$ can be partitioned into three disjoint non-empty sets $A,B,C$ (where the ordering of the parts matters) such that every vertex in $A$ beats every vertex in $B$, every vertex in $B$ beats every vertex in $C$ and every vertex in $C$ beats every vertex in $A$. 

\begin{proposition}\label{propmain}
Let $E(n,\delta,d)$ be the event that there exists a tripartite tournament $T\in\Tourn(n)$ with vertex partition $A,B,C$ such that $|A|,|B|\leq \delta n$ and $\Theta^{\rand}_{d,n}(T)\in A$. Then for any $\delta>0$ there exist an integer $d_0>0$ and infinitely many $n$ with the property that for all $d\geq d_0$ we have 
$$\mathbb{P}\left(E(n,\delta,d)\right) > 1-\delta.$$
\end{proposition} 

Theorem~\ref{theorem} follows easily from Proposition~\ref{propmain}.
\begin{proof}[Proof of Theorem~\ref{theorem}]
Let $\delta = \eps/2$ and let $T\in\Tourn(n)$ be a tripartite tournament  with vertex partition $A,B,C$ such that $|A|,|B|\leq \delta n$. Note that as every vertex in $C$ beats every vertex in $A$, every vertex in $A$ has out-degree at most $2\delta n$. Hence if $\Theta$ is a voting tree with $\Theta(T)\in A$ then in particular we have that $\guar(\Theta)\leq 2\delta n=\eps n$. By Proposition~\ref{propmain} there exist infinitely many $n$ with the property that $$\liminf_{d\rightarrow\infty}\mathbb{P}\left(E(n,\delta,d)\right)\geq 1-\delta,$$ and hence $$\liminf_{d\rightarrow\infty}\mathbb{P}\left(\guar\left(\Theta^{\rand}_{d,n}\right)\leq \eps n\right)\geq 1-\eps.$$
\end{proof}

Now we are ready to prove the main proposition.
\begin{proof}[Proof of Proposition~\ref{propmain}]
Fix an arbitrarily small rational $\eps>0$. Let $a_1,\ldots,a_6$ be the six points given by Lemma~\ref{sixpoints}. Let $n$ be a common multiple of all $18$ denominators of their coordinates. Given a tripartite tournament $T$, the distribution at layer $d$ of a $d$-RPT on input $T$ corresponds to the point $(|A|/n,|B|/,|C|/n)$ in $\Delta$. In this sense, the six points $a_1,\ldots,a_6$ correspond to $n$-vertex tripartite tournaments of the form $A\rightarrow B\rightarrow C\rightarrow A$ with $|A|,|B|\leq \eps n$. By Lemma~\ref{sixpoints} there exists a $d_0$ such that for all $d>d_0$, the probability that the winner in one of these tournaments is in $A$ is at least $1-\eps$. Hence with probability $1-\eps$ the guarantee of the random voting tree is less than $2\eps n$. 
\end{proof}

\section{Proofs of Propositions~\ref{99decr}, \ref{epsclosecompact}, \ref{skipcorner} and~\ref{skipcorner2}}

Some of the following proofs are implicitly present in~\cite{misi}. To make the paper self-contained we give full proofs below. 

\begin{proof}[Proof of Proposition~\ref{99decr}]
It suffices to prove the inequality
$$(1+y-z)(1+z-x)(1+x-y)\leq 1-\eps^3.$$
Let $a=1+y-z$, $b=1+z-x$ and $c=1+x-y$, and since $(x,y,z)\in M(\eps)$ we can assume that $a\geq 1+\frac{\eps}{2}$,~ $b\leq 1$ and $c\leq a$. It follows from a strengthening of the AMGM inequality (see e.g.~\cite{amgm}) that
\begin{equation*}
\begin{split}
\frac{a+b+c}{3}-\sqrt[3]{abc}&\geq\frac{1}{3}\left(\sqrt{a}-\frac{1}{3}\left(\sqrt{a}+\sqrt{b}+\sqrt{c}\right)\right)^2\\
&\geq \frac{1}{3}\left(\frac{1}{3}\sqrt{a}-\frac{1}{3}\sqrt{b}\right)^2\geq \frac{1}{27}\left(\frac{\eps}{5}\right)^2,
\end{split}
\end{equation*}
and the result follows.
\end{proof}

\begin{proof}[Proof of Proposition~\ref{epsclosecompact}]
We will assume $\eps < 2^{-100}$. Let $C:=\eps^{-15}$. By Proposition~\ref{99decr} we have $$\phi\left(\VV^{C}(a)\right)\leq \eps^{12}.$$ Let $b:=\VV^C(a)$ and suppose that $b=(x,y,z)$ is not $\eps$-close to any vertex. Then it has a coordinate, wlog $x$ with $\eps<x<1-\eps$. Since $\phi(b)<\eps^{12}$ we have either $y\leq\eps^4$ or $z\leq\eps^4$. We will assume $z<\eps^4$, the case of $y<\eps^4$ is very similar.

Apply $\VV$ repeatedly to obtain $\VV^i(b)=(x_i,y_i,z_i)$ for $i=0,1,2,\ldots$. Let $N_1\geq 0$ be the least integer with $x_i>0.1$. 
 Then for all $i\in\{0,1,\ldots,N_1-1\}$ we have $z_{i+1}= z_i^2+2z_ix_i\leq (0.2+z_i)z_i$ and since initially $z_0<0.1$ it follows that $z_i^2<z_{i+1}<z_i/2$. Since $y_i+z_i+x_i=1$ for all $i$, we have  for all $i\leq N_1-1$ that $y_i\geq 0.8$. So for $i\leq N_1-1$ we have $x_{i+1}=x_i^2+2x_iy_i\geq 1.5 x_i$. Since $x_i\leq 1$ for all $i$, we have $N_1\leq 10\log \eps^{-1}$. Note that we have
$$0.1<x_{N_1}<0.2, \quad\text{and}\quad z_{N_1}\leq z_0<\eps^4.$$

Let $N_2$ be the least integer larger than $N_1$ such that $x_{N_2}>0.9$. We claim $N_2\leq N_1+100$. First observe that for all $i\in\mathbb{N}$ we have $z_{i+1}\leq 2z_i$ and hence for all $i\in[N_1,N_1+100]$ we have $z_i\leq \eps^2$. If for some $i$ we have $x_i\leq 0.9$, $z_i\leq \eps^2$ then $y_i\geq 0.09$ and hence $x_{i+1}=x_i(1+y_i-z_1)\geq x_i\cdot 1.05$. Since $x_{N_1}>0.1$ and $0.1\cdot 1.05^{100}>0.9$ this proves that $N_2-N_1\leq 100$.

 Now we have
$$0.9<x_{N_2}<0.995, \quad z_{N_2}\leq \eps^2,\quad\text{and so}\quad 0.004<y_{N_2}<0.1.$$

Let $N_3$ be the least integer larger than $N_2$ such that $x_{N_3}\geq 1-\eps$, and let $N'$ be the least integer larger than $N_2$ such that $z_{N'}\geq \eps/2$. We will first show that $N_3\leq N'$. Note that when $i\in [N_2,N'-1]$ we have $z_{i}<\eps/2$ and so $y_{i+1}\leq y_{i}/2$. As for any $i$ we have $z_{i+1}\leq 2z_i$ we know that $\eps/4\leq z_{N'-1}\leq \eps/2$. Since $\phi\left(\left(x_{N'-1},y_{N'-1},z_{N'-1}\right)\right)\leq \phi\left(\left(x_0,y_0,z_0\right)\right)\leq \eps^8$, we conclude that $x_{N'-1}y_{N'-1}\leq \eps^6$. Now assume for contradiction that $N'\leq N_3$. Then for all $i\in[N_2,N'-1]$ we have $y_i>\eps/2$ (as $N'\leq N_3$). Since for $i\in [N_2,N'-1]$ the $y$ coordinate is decreasing we conclude $x_{N'-1}=1-z_{N'-1}-y_{N'-1}\geq 1-\eps-0.1\geq 0.8$. Now the inequality $x_{N'-1}y_{N'-1}\leq \eps^6$ implies $y_{N'-1}\leq \eps^5$, a contradiction. Hence we have shown that $N_3\leq N'$.

This implies that for $i\in[N_2,N_3-1]$ we have $y_i\leq 0.1$ and $z_i\leq \eps/2$ and so since $y_{i+1}=\leq 0.5y_i$ in this range, we have  $$N_3\leq N_2+ 10\log \eps^{-1}.$$ Setting $D=C+N_3$ finishes the proof. 
\end{proof}

\begin{proof}[Proof of Proposition~\ref{skipcorner}]
Given $0<\eps<0.1$ and $D\in\mathbb{N}$, we claim that $$\eps_1:=\left(\frac{\eps}{2}\right)^{2^{D}}$$ satisfies the statement of Proposition~\ref{skipcorner}. Indeed, fix an $\mathbf{a}=(x_0,y_0,z_0)\in\Delta$ that is $\eps_1$-close to the $xy$ side but not $\eps$-close to the $x$ corner. Denote the coordinates of $\VV^i(\mathbf{a})$ as $(x_i,y_i,z_i)$. Then $x_0\leq 1-\eps$ and $z_0\leq \eps^{2^{D}}$, implying that $y_0\geq \eps/2$.  The inequality $d\geq D$ follows from our choice of $\eps_1$ and the fact that for any $i$ we have $y_{i+1}\geq y_i^2$. Indeed, as $y_0\geq \eps/2$ this shows that for any $i$, we have $y_i\geq (\eps/2)^{2^i}$, and hence $y_i\leq \eps_1$ implies $i\geq D$.
\end{proof}

\begin{proof}[Proof of Proposition~\ref{skipcorner2}]
Given $0<\eps<0.1$ and $D\in\mathbb{N}$, we claim that $$\eps_1:=\frac{\eps}{2^{2D}}$$ satisfies the claim.  Indeed, fix an $\mathbf{a}=(x_0,y_0,z_0)\in\Delta$ that is $\eps_1$-close to the $x$ corner of $\Delta$ and denote the coordinates of $\VV^i(\mathbf{a})$ as $(x_i,y_i,z_i)$ as before. If $d\in\mathbb{N}$ is such that $V^d(\mathbf{a})$ is not $\eps$-close to the $x$ corner then $z_{d}\geq \eps/2$ or $y_d\geq \eps/2$. Initally we had $y_0,z_0\leq \eps_1=\eps 2^{-2D}$ and each coordinate at most doubles in every step, implying $d\geq D$ as required.
\end{proof}

\section*{Acknowledgements}

We are very grateful to Shagnik Das, Christopher Kusch and Tam\'as M\'esz\'aros for many helpful discussions.

\end{document}